\documentclass[10pt]{amsart}
\usepackage{amssymb,amsmath,amsfonts,amsthm}
\usepackage[utf8]{inputenc}
\usepackage[english]{babel}
\usepackage[dvipsnames]{xcolor}
\usepackage{mathtools}
\usepackage{soul}
\mathtoolsset{showonlyrefs}

\newtheorem{theorem}{Theorem}[section]

\newtheorem{proposition}{Proposition}[section]
\newtheorem{definition}{Definition}[section]
\newtheorem{remark}{Remark}
\newtheorem{example}{Example}

\def\<{\langle}
\def\>{\rangle}
\begin{document}
\title[Periodic symbols]{Pseudodifferential operators with completely periodic symbols}
\author[G. Garello]{Gianluca Garello}
\address{Mathematics Department, University of Turin, via Carlo Alberto 10, I-10123 Torino, Italy. (ORCID: 0000-0002-8636-2998 )}
\email{gianluca.garello@unito.it}
\author[A. Morando]{Alessandro Morando}
\address{DICATAM, University of Brescia, via Valotti 9, I-25133 Brescia, Italy. (ORCID: 0000-0002-6800-7490)}
\email{alessandro.morando@unibs.it}

\begin{abstract}
Motivated by the recent paper of Boggiatto-Garello in J. Pseudo-Differ. Oper. Appl. \textbf{11} (2020), 93-117, where a Gabor operator is regarded as pseudodifferential operator with symbol $p(x,\omega)$ periodic on both the variables, we study the continuity and invertibility, on general time frequency invariant spaces, of pseudodifferential operators with completely periodic symbol and general $\tau$ quantization.
\end{abstract}
\maketitle

\tableofcontents
\section{Introduction.}\label{INT}
Consider the formal expression of the Gabor operator: $Sf=\sum\limits_{h,k\in \mathbb Z^d} (f, g_{h,k})_{_{L^2}}g_{h,k}$,  $g_{h,h}(t)=e^{2\pi i\beta k\cdot t}g(t-\alpha h)$, $\alpha, \beta\in\mathbb R_+$, and that of the pseudodifferential operator with Kohn-Nirenberg quantization: $ \displaystyle {a(x,D) \varphi(x)=  \iint e^{2\pi i \omega\cdot(x-t)}a(x,\omega) f(t)\, dt\, d\omega}$, where $f \in \mathcal S(\mathbb R^d)$, $a(x,\omega)\in \mathcal S'(\mathbb R^{2d})$ and the integration is intended in distribution sense.  In the recent work \cite{BogGar2020} it is proven that  $S=a(x,D)$, where
\begin{equation}\label{INT01}
 a(x,\omega)=\sum_{h,k \in \mathbb Z^d} e^{-2\pi i (x-\alpha h)\cdot (\omega-\beta k)}g(x-\alpha h)\bar{\hat g} (\omega-\beta k),
\end{equation}
with suitable decay conditions at infinitive of $g$, $\hat g$, and convergence in $L^\infty(\mathbb R^{2d})$. Then using the Calder\'on - Vaillancourt Theorem about $L^2$ continuity of pseudodifferential operators, see \cite{HW1}, one can prove that for $\alpha+\beta$ less than a suitable positive constant $C_g$, depending only on the decay at infinitive of $g, \hat g$ and some of their derivatives, the Gabor operator is invertible as a bounded linear operator on $L^2(\mathbb R^d)$. Then as well known in frame theory, the Gabor system $\mathcal \{g_{h,k}\}_{h,k\in \mathbb Z^d}$ realizes to be a frame.\\
The literature about Gabor frame theory is wide, we quote here only the  monographies \cite{Coh95-1}, \cite{GRO1}, \cite{Hei11}, \cite{Chr16}. Among others, the problem of finding conditions on  the parameters $\alpha, \beta$ in order to obtain Gabor frames is a challenging one, see for example \cite{Dai-Sun16}, \cite{GroKop19} and the references therein.\\
Notice now that the symbol in \eqref{INT01} is completely periodic, with period $\alpha$ with respect to the spatial variable $x$ and period $\beta$ with respect to $\omega$. For all the reasons listed above, we think  it should be of  some relevance to develop the study of pseudodifferential operators with completely periodic symbols, their continuity and possible invertibility in $L^2$ or more general function spaces.\\
Concerning symbols independent of $x$, that is Fourier multipliers, we quote the papers \cite{deL65}, \cite{Iga69}, where the periodic case in considered.\\
Wider is the literature concerning the pseudodifferential operators on compact Lie groups, see the  fundamental book of Ruzhansky-Turunen \cite{RuzTur}, which have as particular case symbols periodic in $x$ and discrete (non periodic) in $\omega$. Also interesting is the reversed case where the symbols are discrete in $x$ and periodic in $\omega$; the related operators are called in this case "pseudo-difference" operators, see  \cite{BoKiRu20}, \cite{KumMon2023}.
About pseudodifferential operators on generalized spaces, e.g. modulation spaces, we refer to the following papers \cite{Tof04-1}, \cite{Tof04-2}, \cite{CorNic10}, \cite{BogDedOli2010}, \cite{PTT2}, \cite{PTT1}, \cite{CaTo17}, \cite{DeTra18}, \cite{CorDeTra19}, \cite{CorNicTra19}.\\
The plan of the paper is the following: in \S \ref{SBT} we give the notations and definitions; then we review some basic facts about periodic distributions with respect to a general invertible matrix and their Fourier transform. In \S \ref{SPDO} we introduce the pseudodifferential operators with general $\tau$ quantization and for the case of periodic symbols  we provide a representation formula, obtained by linear combination of time frequency shift operators. 
At the end, respectively in \S \ref{CONT} and \S \ref{INVERT} we set the results of continuity and invertibility on general families of time frequency shift invariant spaces.
The Appendix \ref{PER-DIST} is devoted to give some technicalities in order to compare periodic distributions on $\mathbb R^n$ and distribuitions on the $n$ dimensional torus $\mathbb T^n$.\\

\section{Preliminaries}\label{SBT}
\subsection{Notations and basic tools}
In whole the paper we will use the following notations and tools:
\begin{itemize}
\item $\mathbb R^n_0=\mathbb R^n\setminus\{0\}$, $\mathbb Z^n_0=\mathbb Z^n\setminus\{0\}$;
\smallskip

\item $\langle x\rangle=\sqrt{1+\vert x\vert^2}$, where $\vert x\vert$ is the Euclidean norm of $x\in \mathbb R^n$;
\item $x\cdot \omega=\langle x,\omega\rangle=\sum_{j=1}^n x_j\omega_j$; \quad $x,\omega\in \mathbb R^n$;
\item $(f,g)=\int f(x)\bar g(x)\, dx$ is the inner product in $ L^2(\mathbb R^n)$;
\item $\mathcal F f(\omega)=\hat f(\omega)=\int f(x) e^{-2\pi i x\cdot \omega} \, dx$  the Fourier transform of $f\in \mathcal S(\mathbb R^n)$, with the well known extension to $u\in \mathcal S'(\mathbb R^n)$.
\end{itemize}
 The {\em polynomial weight function} $v$ is defined  for some $s\ge 0$ by
\begin{equation}\label{pol}
v(z)=(1+\vert z\vert^2)^{s/2}\,,\quad\forall\,z\in\mathbb R^n\,.
\end{equation}
A non negative measurable function $m=m(z)$ on $\mathbb R^n$ is said to be a \textit{polynomially moderate} (or \textit{temperate}) weight function if there exists a positive constant $C$ such that
\begin{equation}\label{SOPS1}
m(z_1+z_2)\leq C v(z_1)m(z_2) \quad \text {for all}\, z_1, z_2\in \mathbb R^n.
\end{equation}

For other details about weight functions see \cite[\S 11.1]{GRO1}.

In the following we will use in many cases the matrix in $GL(2d)$ 
\begin{equation}\label{SYMPMAT}
\mathcal J=\left(
\begin{array}{cc}
0&-I\\
I&0
\end{array}
\right),
\end{equation} 
which defines the \textit{symplectic} form, see \cite[\S 9.4]{GRO1},
\begin{equation}\label{SYMP}
[z_1, z_2]:= \langle z_1, \mathcal J z_2\rangle= x_2\cdot\omega _1-x_1\cdot\omega_2\quad , \quad z_1=(x_1, \omega_1),\, z_2=(x_2, \omega_2) \in \mathbb R^{2d}. 
\end{equation} 

\subsection{Time frequency shifts (tfs)}\label{SubTF}
For $z=(x,\omega)\in \mathbb R^{2d} $ we define the operators:
\begin{align}
&T_x f(t)=f(t-x)  &\text{(translation)};\label{TR}\\
&M_\omega f(t)=e^{2\pi i \omega \cdot t} f(t) &\text{(modulation)},\label{MO}\\
&\pi_z f=M_\omega T_x f  &\text{(time frequency shift)},\label{TFS}
\end{align}
with suitable extension to distributions in $\mathcal D'(\mathbb R^{d})$.\\
For $u\in\mathcal S'(\mathbb R^{d})$, $z=(x,\omega)\in \mathbb R^{2d}$, the next properties easily follow:
\begin{align}
& T_x M_\omega u=e^{-2\pi i x\cdot\omega}M_\omega T_x u \label{SCAMBIO},\\
&\mathcal F(T_x u)=M_{-x}\mathcal Fu, &&\mathcal F^{-1}(T_x u)=M_x\mathcal F^{-1} u \label{TR1},\\
&\mathcal  F(M_\omega u)=T_\omega\mathcal Fu, &&\mathcal F^{-1}(M_\omega u)=T_{-\omega}\mathcal F^{-1} u\label{MO2},\\
&\mathcal F(\pi_z u)=e^{2\pi i x\cdot \omega}\pi_{\mathcal J^T z}\mathcal Fu; &&\mathcal F^{-1}(\pi_z u)=e^{2\pi i x\cdot \omega}\pi_{\mathcal J z}\mathcal F^{-1} u.\label{TM}
\end{align}


\subsection{Modulation spaces}
\begin{definition}\label{stft-def}
For a fixed nontrivial function $g$ the short-time Fourier transform (or Gabor transform)  of a function $f$ with respect to $g$ is defined as
\begin{equation}\label{stft}
V_gf(z):=(f,\pi_z g)=\int_{\mathbb R^d}f(t)e^{-2\pi i \omega\cdot t}\overline{g(t-x)}dt\,,\quad\mbox{for}\, \,z=(x,\omega)\in\mathbb R^{2d}\,,
\end{equation}  
whenever the integral can be considered, also in weak distribution sense.
\end{definition}
When $f,g\in L^2(\mathbb R^d)$ ,  $V_gf$ is a uniformly continuous function on $\mathbb R^{2d}$, $V_gf\in L^2(\mathbb R^{2d})$ and
$
\Vert V_gf\Vert_{L^2}=\Vert f\Vert_{L^2}\Vert g\Vert_{L^2}
$.
See e.g. \cite[\S3]{GRO1}.
\begin{definition}\label{M_def}
For a fixed $g\in\mathcal S(\mathbb R^d)\setminus\{0\}$ and $p,q\in[1,+\infty]$, the $m-$weighted modulation space $M^{p,q}_m(\mathbb R^{d})$ consists of all tempered distributions $f\in\mathcal S^\prime(\mathbb R^{d})$ such that 
\begin{equation}\label{M-norm}
\Vert f\Vert_{M^{p,q}_m}:=\left(\int_{\mathbb R^d}\left(\int_{\mathbb R^d}\vert V_g f(x,\omega)\vert^p m(x,\omega)^p dx\right)^{q/p}d\omega \right)^{1/q}< +\infty\, ,
\end{equation}
(with expected modification in the case when at least one among $p$ or $q$ equals $+\infty$).
\end{definition}
The definition of the space $M^{p,q}_m$ is independent of the choice of the window $g$, different windows $g$ provide equivalent norms and $M^{p,q}_m$ turns out to be a Banach space. In the case of $p=q$ we denote $M^p_m:=M^{p,p}_m$, when $m(x,\omega)\equiv 1$ we write $M^{p,q}$. 

For more details about modulation  spaces see \cite[\S 6.1, \S11]{GRO1}.

\subsection{Periodic distributions}\label{SubPD}
We say that a distribution $u\in \mathcal D'(\mathbb R^n)$ is periodic (of period 1) if
\begin{equation}\label{PD}
T_\kappa u=u\quad \text{for any}\,\, \kappa\in \mathbb Z^n.
\end{equation}
Notice that $u$ is in this case a tempered distribution in $\mathcal S'(\mathbb R^n)$, then its Fourier transform $\hat u$ can be considered. Moreover it can be shown that
\begin{equation}\label{HS1} 
\hat u= \sum_{\kappa\in\mathbb Z^n} c_\kappa (u)\delta_\kappa,
\end{equation}
where the series converges in $\mathcal S'(\mathbb R^n)$,
\begin{equation}\label{HS2} 
c_\kappa(u):=\langle u, \phi e^{-2\pi i \langle \cdot, \kappa \rangle}\rangle= \widehat {u\phi}(\kappa),\\
\end{equation}
and $\phi\in C^\infty_0(\mathbb R^n)$ satisfies 
\begin{equation}\label{PHI}
\sum_{\kappa\in \mathbb Z^n} \phi (x-\kappa)=1.
\end{equation}
For the details see H\"ormander \cite[\S 7.2]{HOR0}.  Now by a straightforward application of Fourier inverse transform we obtain
\begin{equation}\label{HS3}
u= \sum_{\kappa\in\mathbb Z^n} c_\kappa(u) e^{2\pi i \langle \cdot, \kappa\rangle},
\end{equation}
 with convergence in $\mathcal S'$ and  $c_\kappa(u)$ defined in \eqref{HS2}.\\
Notice that a general periodic distribution $u$ can be regarded as a distribution on the torus $\mathbb T^n=\mathbb R^n/\mathbb Z^n$, that is a linear continuous  form on $C^\infty(\mathbb T^n)$. In the following  $\mathcal D'(\mathbb T^n)$ will be the topological dual space of $C^\infty(\mathbb T^n)$. Thus the coefficients in the expansion \eqref{HS3} can be regarded as the Fourier coefficients of $u$,  namely
\begin{equation}\label{HS4}
c_\kappa (u)=\langle u, e^{-2\pi i \langle \cdot, \kappa\rangle } \rangle_{_{\mathbb T^n}},
\end{equation}
where $\langle \cdot, \cdot\rangle_{\mathbb T^n}$ denotes the duality pair between $\mathcal D'(\mathbb T^n)$ and $C^\infty (\mathbb T^n)$. 
Agreeing with \eqref{HS4}, for  $f\in L^1(\mathbb T^n)$, using \eqref{PHI}, we get 
\begin{equation}\label{HS5}
c_\kappa(f)=\int_{\mathbb R^n} f(x)\phi(x) e^{-2\pi i x\cdot \kappa}\, dx=\int_{[0,1]^n}  f(x) e^{-2\pi i x\cdot \kappa} \, dx.
\end{equation} 
In Appendix \ref{PER-DIST} we clarify how to make rigorous the above calculations in $\mathcal D'(\mathbb T^n)$, see in particular \eqref{APP4}.\\ 
For a detailed discussion about distributions on the torus one can see  the book of M. Ruzhansky and V. Turunen \cite{RuzTur}.
\smallskip

Consider now and in the whole paper $L=(a_{ij})\in GL(n)$, the space of invertible matrices of size $n\times n$.\\
For $\mathbb T^n_L:=\mathbb R^n/L\mathbb Z^n$, we still identify the set of $L$-periodic distributions, that is $u\in\mathcal D'(\mathbb R^n)$ such that, for any $\kappa \in \mathbb Z^n, T_{L\kappa}u=u$, with the space $\mathcal D'(\mathbb T^n_L)$ of linear continuous forms on $C^\infty(\mathbb T^n_L)$. Notice that also in this case $\mathcal D'(\mathbb T^n_L)\subset\mathcal S'(\mathbb R^n)$. \\
For any $u\in \mathcal D'(\mathbb T^n_L)$ observe that  $v=u(L\cdot)$ is a  $1$-periodic distribution. Applying then  the  Fourier expansion
$v=\sum_{\kappa \in \mathbb Z^n}c_{\kappa}(v) e^{2\pi i \langle \cdot,\kappa\rangle}$, $c_{\kappa}(v)=\langle v, e^{-2\pi i \langle \cdot, \kappa\rangle}\rangle_{\mathbb T^n}$, we obtain 
\begin{equation}\label{FL1}
u=v(L^{-1} \cdot)=\sum_{\kappa\in\mathbb Z^n} c_{\kappa}(v) e^{2\pi i \langle k,L^{-1}\cdot\rangle}=\sum_{\kappa\in\mathbb Z^n} c_{\kappa}(v) e^{2\pi i \langle L^{-T} k,\cdot\rangle},
\end{equation}
where $L^{-T}:=(L^{-1})^T$ denotes the transposed of the inverse matrix of $L$ and
\begin{equation}\label{FL2}
c_{\kappa}(v)=\langle  u(L\cdot), e^{-2\pi i \langle  \kappa, \cdot\rangle}   \rangle_{\mathbb T^n}=
\frac{1}{\vert\textup{det}L\vert}\langle  u, e^{-2\pi i \langle  L^{-T}\kappa, \cdot\rangle}   \rangle_{\mathbb T^n_L}.
\end{equation}
Thus we obtain for any $u\in \mathcal D'(\mathbb T^n_L)$ the Fourier expansion
\begin{equation}\label{FL5}
u= \sum_{\kappa\in\mathbb Z^n} c_{\kappa,L}(u)e^{2\pi i \langle L^{-T} k,\cdot\rangle},
\end{equation}
with the Fourier coefficients
\begin{equation}\label{FL6}
c_{\kappa,L}(u):=c_\kappa(u(L\cdot))= \frac{1}{\vert\textup{det} L\vert}\langle  u, e^{-2\pi i \langle  L^{-T}\kappa, \cdot\rangle}   \rangle_{\mathbb T^n_L}.
\end{equation}
For short in the following we set $c_\kappa(u)=c_{\kappa, L}(u)$.\\
Consider $L^p(\mathbb T^n_L)$, $1\leq p <\infty$,  the set of measurable L-periodic functions on $\mathbb R^n$ such that $\Vert f\Vert_{L^p(\mathbb T^n_L)}=\int_{L[0,1]^n} \vert f(x)\vert ^p\, dx<\infty$, with obvious modification for the definition of $L^\infty(\mathbb T^n_L)$.\\
Then for $f\in L^1(\mathbb T^1_L)$ the following:\\
\begin{equation}\label{FS1}
f(x)=\sum_{\kappa\in \mathbb Z^n} c_\kappa (f) e^{2\pi i L^{-T}\kappa\cdot x}
\end{equation}
holds with convergence  in $\mathcal S'(\mathbb R^n)$, and 
\begin{equation}\label{FS2}
c_\kappa(f)=\frac{1}{\vert \textup{det}L\vert }\int_{L[0,1]^n} e^{2\pi i L^{-T}\kappa\cdot x} f(x)\, dx\\
\end{equation} 
\begin{remark}
It can be useful to write the Fourier expansion of $u\in\mathcal D'(\mathbb T^n_L)$ in terms of the lattice  $\Lambda=L \mathbb Z^n$:
\begin{equation}\label{FL3}
u=\sum_{\mu\in \Lambda^{\bot}}\hat u(\mu)e^{2\pi i\langle  \mu,\cdot \rangle},
\end{equation}
with \begin{equation}\label{FL4}
\hat u(\mu):=\frac{1}{\textup{vol}(\Lambda)}\langle  u, e^{-2\pi i \langle  \mu, \cdot\rangle}   \rangle_{\mathbb T^n_L}.
\end{equation}
Here $\Lambda^\perp:=L^{-T}\mathbb Z^n$ and  $\textup{vol}(\Lambda):=\vert \textup{det} L\vert=\text{meas} \,(L[0,1]^n)$  are respectively called dual lattice and volume of $\Lambda$.
\end{remark}

\begin{example}\label{EX1}
For $\alpha=(\alpha_1, \dots, \alpha_n)\in \mathbb R^n$, $\alpha_j>0$, let us consider the diagonal matrix  $A\in GL(n)$, together with its inverse
\begin{equation}\label{DM}
A=\left(
\begin{array}{cccc}
\alpha_1&\dots&0&\\
\vdots&\ddots&\vdots&\\
0&\dots&\alpha_n&
\end{array}
\right)
\quad ; \quad
A^{-1}=\left(
\begin{array}{cccc}
\frac{1}{\alpha_1}&\dots&0&\\
\vdots&\ddots&\vdots&\\
0&\dots&\frac{1}{\alpha_n}&
\end{array}
\right)
\end{equation}
and introduce for $\kappa\in \mathbb Z^n$ the following notations, $\alpha k:=A\kappa= (\alpha_1 k_1, \dots, \alpha_n k_n)$; $\frac{\kappa}{\alpha}:=A^{-1}k=
\left( \frac{\kappa_1}{\alpha_1}, \dots, \frac{\kappa_n}{\alpha_n}\right)$, $\prod\alpha:=\prod_{j=1}^n \alpha_j$ .  Consider now  an  $A$-periodic function $f$ which satisfies $\int_0^{\alpha_1}\dots\int _{0}^{\alpha_n}  \vert f(x)\vert \, dx_1\dots dx_n<+\infty$, then directly from \eqref{FS1}, \eqref{FS2} we obtain 
\begin{align}
&f(x)=\sum_{\kappa\in \mathbb Z^n} c_k(f) e^{2\pi i \frac{k}{\alpha}\cdot x},& \label{ES1.1}\\
\text{where}\\
\\
&c_k(f)= \frac{1}{\prod \alpha}\int_0^{\alpha_1}\dots\int_0^{\alpha_n} f(x) e^{-2\pi i \frac{\kappa}{\alpha}\cdot x}\, dx_1\dots dx_n.&\label{ES1.2}
\end{align}
\end{example}

 \section{Pseudodifferential Operators with periodic symbol}\label{SPDO}
We say $\tau$ pseudodifferential operator, $0\leq \tau\leq 1$, with symbol $p(z)=p(x,\omega) \in \mathcal S'(\mathbb R^{2d})$,  the operator acting from $\mathcal S(\mathbb R^{d})$ to $\mathcal S'(\mathbb R^{d})$ defined  by
\begin{equation}\label{PS1}
\textup{Op}_{\tau}(p)u(x):=\int_{\mathbb R^d_\omega}\int_{\mathbb R^d_y}e^{2\pi i (x-y)\cdot \omega}p\left((1-\tau) x+\tau y, \omega\right) u(y)\, dy\,d\omega, \quad u\in \mathcal S(\mathbb R^d).
\end{equation}
The formal integration must be understood in distribution sense.
For the definition and development of pseudodifferential operators see the basic texts  \cite{Shu87}, \cite{Hor94}.\\


For $I$ and $0$  respectively the identity and null matrices of dimension $d\times d$, let us introduce the $d\times 2d$ matrices
\begin{equation}\label{MAT12}
I_1=(I,0), \quad I_2=(0,I)
\end{equation}

\begin{proposition}\label{PROPPST}
Consider $p\in \mathcal D'(\mathbb T^{2d}_L)$, $L\in GL(2d)$.
Then for any $0\leq \tau\leq 1$ and $u\in \mathcal S(\mathbb R^d)$ we can write
\begin{equation}\label{PST2}
\textup{Op}_\tau(p)u=\sum_{\kappa\in \mathbb Z^{2d}}c_\kappa(p) e^{{2\pi i \tau}\langle I_2 L^{-T}\kappa\,,\,  I_1L^{-T}\kappa\rangle}\pi_{_{\mathcal JL^{-T} \kappa}}u,
\end{equation}
where $c_\kappa(p)$ are the Fourier coefficients defined in \eqref{FL6} and $\mathcal J$ the matrix introduced in \eqref{SYMPMAT}.
\end{proposition}
\begin{proof}
Using \eqref{FL5}, \eqref{FL6} we perform the Fourier expansion of the symbol $p$
\begin{equation}\label{PST3.1} 
p=\sum_{\kappa\in \mathbb Z^{2d}} c_\kappa(p) e^{2\pi i \langle L^{-T}\kappa, \cdot\rangle}
\end{equation}
with convergence in $\mathcal S'(\mathbb R^{2d})$.
Then for any $u\in \mathcal S(\mathbb R^d)$ we get
\begin{equation}\label{PST4}
\textup{Op}_\tau(p)u(x)=\iint e^{2\pi i (x-y)\cdot\omega}\sum_{\kappa\in\mathbb Z^{2d}} c_\kappa(p) e^{2\pi i 
L^{-T}\kappa \cdot \left((1-\tau) x+\tau y, \omega \right)} u(y)\, dy\, d\omega.
\end{equation}
Considering the decomposition 
\begin{equation}\label{LDEC}
L^{-T}=\left(
\begin{array}{c}
I_1 L^{-T}\\
I_2 L^{-T}
\end{array}
\right),
\end{equation} 
we obtain 
\begin{equation}\label{PST6}
L^{-T}\kappa \cdot\left( (1-\tau)x+\tau y, \omega  \right)
= I_1 L^{-T} \kappa\cdot\left( (1-\tau)x+ \tau y\right)+ I_2 L^{-T}\kappa\cdot\omega.
\end{equation}
Let us set for simplicity of notation $L_j^{-T}=I_j L^{-T}$, $j=1,2$. Then in view of convergence in $\mathcal S'$ and formal integration in distribution sense it follows
\begin{equation}\label{PST7}
\begin{array}{l}
\begin{array}{ll}
\textup{Op}_\tau(p)u(x)=\displaystyle
\sum\limits_{\kappa\in \mathbb Z^{2d}}c_\kappa(p) \int\!\!\!\int e^{2\pi i (x-y)\cdot \omega}
&e^{2\pi i  L^{-T}_1\kappa\cdot \left( (1-\tau)x+\tau y\right)}\times\\
& \times
e^{2\pi i L^{-T}_2\kappa \cdot \omega} u(y)\, dy\, d\omega=
\end{array}
\\
\\
=\displaystyle\sum\limits_{\kappa\in\mathbb Z^{2d}}c_\kappa(p)\int 
e^{2\pi i x\cdot\omega}e^{2\pi i L^{-T}_2\kappa\cdot \omega}
\int e^{-2\pi i y\cdot\omega}e^{2\pi i L^{-T}_1\kappa\cdot \left( (1-\tau)x+\tau y\right)} u(y)\, dy\, d\omega=
\\
\\
=\displaystyle\sum\limits_{\kappa\in \mathbb Z^{2d}} c_\kappa(p)e^{2\pi i L^{-T}_1\kappa\cdot \left( (1-\tau)x\right)}\int e^{2\pi i \left( x+ L^{-T}_2\kappa \right)\cdot \omega}
\int e^{-2\pi i \left(\omega-\tau L^{-T}_1\kappa \right)\cdot y} u(y)\, dy\,d\omega=
\\
\\
=\displaystyle \sum\limits_{\kappa\in\mathbb Z^{2d}}c_\kappa(p)e^{2\pi i (1-\tau) L^{-T}_1\kappa\cdot x}\int e^{2\pi i \left(x+ L^{-T}_2\kappa \right)\cdot\omega}
\hat u\left(\omega-\tau L^{-T}_1\kappa\right)\, d\omega=
\\
\\
=\displaystyle\sum\limits_{\kappa\in\mathbb Z^{2d}}c_\kappa(p)e^{2\pi i (1-\tau) L^{-T}_1\kappa\cdot x}\int 
e^{2\pi i \left(x+L^{-T}_2\kappa\right)\cdot \omega} T_{\tau L^{-T}_1\kappa}\hat u(\omega)\, d\omega=
\\
\\
=\displaystyle\sum\limits_{\kappa\in\mathbb Z^{2d}}c_\kappa(p)e^{2\pi i (1-\tau) L^{-T}_1\kappa\cdot x}\int e^{2\pi i\left( x+L^{-T}_2 \kappa\right)\cdot\omega}
\widehat{M_{\tau L^{-T}_1\kappa} u}(\omega)\, d\omega=
\\
\\
=\displaystyle\sum\limits_{\kappa\in\mathbb Z^{2d}}c_\kappa(p)e^{2\pi i (1-\tau) L^{-T}_1\kappa\cdot x}
\left(M_{\tau L^{-T}_1\kappa} u\right)
\left(x+L^{-T}_2 \kappa\right)=
\\
\\
\displaystyle=\sum\limits_{\kappa\in\mathbb Z^{2d}}c_\kappa(p)e^{2\pi i (1-\tau) L^{-T}_1\kappa\cdot x}
\,T_{-L^{-T}_2\kappa}M_{\tau L^{-T}_1\kappa}u(x)=
\\
\\
\displaystyle=\sum\limits_{\kappa\in\mathbb Z^{2d}}c_\kappa(p)M_{(1-\tau) L^{-T}_1 \kappa}T_{- L^{-T}_2\kappa}M_{\tau L^{-T}_1\kappa}u(x).
\end{array}
\end{equation}
The proof ends by observing that, thanks to \eqref{SCAMBIO},
\begin{equation}\label{PST7.1}
\begin{array}{l}
M_{(1-\tau) L^{-T}_1 \kappa} T_{- L^{-T}_2\kappa}M_{\tau L^{-T}_1\kappa}=\\
\\
=e^{-2\pi i \langle - L^{-T}_2 \kappa , \tau L^{-T}_1\kappa \rangle}
M_{(1-\tau) L^{-T}_1 \kappa} M_{\tau L^{-T}_1\kappa} T_{- L^{-T}_2\kappa}=
\\
\\
=e^{2\pi i \tau\langle L^{-T}_2 \kappa , L^{-T}_1\kappa\rangle}
M_{ L^{-T}_1 \kappa}T_{- L^{-T}_2\kappa}=e^{2\pi i \tau\langle L^{-T}_2 \kappa, L^{-T}_1\kappa\rangle}\pi_{\mathcal J L^{-T}\kappa}.
\end{array}
\end{equation}
\end{proof}
\begin{remark}\label{REMPST}
Consider  $\mu= L^{-T}\kappa\in  \Lambda^\perp$. In view of \eqref{PST2} the pseudodifferential operator $\textup{Op}_\tau$ may be written in lattice notation
\begin{equation}\label{REMPST1}
\textup{Op}_\tau(p)=\sum_{\mu\in \Lambda^\perp}\hat p(\mu)e^{2\pi i \tau\langle I_2\mu , I_1\mu\rangle}\pi_{\mathcal J\mu}.
\end{equation}
\end{remark}
\begin{example}\label{EX2}
Let  $a=(a_1, \dots a_d)$, $b=(b_1, \dots, b_d)$ be two vectors in $(\mathbb R\setminus\{0\})^d$. Using the notation in Example \ref{EX1}, we say that a symbol $p\in \mathcal S'(\mathbb R^{2d})$ is $ab$-periodic if, for any $\kappa=(h, k)\in \mathbb Z^{2d}$, $p(\cdot+ah, \cdot+bk)=p(\cdot, \cdot)$. Considering the matrix 
\begin{equation}\label{EX2.1}
L=\left(
\begin{array}{cc}
A&0\\
0&B
\end{array}
\right),
\end{equation}
with $A, B$ diagonal matrices defined as in \eqref{DM}, it is trivial to show that $I_1 L^{-T}\kappa=\frac {h}{a}$ and $I_2 L^{-T}\kappa=\frac {k}{b}$. Then for any $u\in \mathcal D'(\mathbb R^d)$ 
\begin{equation}\label{EX2.2}
\pi_{\mathcal J L^{-T}\kappa}u=
M_{\frac{h}{a}}T_{-\frac{k}{b}}u=
e^{2\pi i \langle \frac{h}{a}, \cdot\rangle}u(\cdot+\frac{k}{b}).
\end{equation}
Thus for any $0\le\tau\le1$ we have
\begin{equation}\label{EX2.3}
\textup{Op}_\tau p(u)= \sum_{(h,k)\in \mathbb Z^{2d}}c_{h,k}(p)e^{2\pi i \tau \langle \frac {h}{a}, \frac{k}{b}\rangle} e^{2\pi i \langle \frac{h}{a}, \cdot\rangle}u(\cdot+ \frac{k}{b}),
\end{equation}
with convergence in $\mathcal S'(\mathbb R^d)$ and $c_{h,k}(p)$ defined by formal integration
\begin{equation}\label{EX2.4}
c_{h,k}(p)=\frac{1}{\vert\prod ab\vert}\int_0^{a_1}\!\!\!\!\!\!\!\dots\!\!\int_0^{a_d}\!\!\!\!\int_0^{b_1}\!\!\!\!\!\!\!\dots\!\!\int_0^{b_d} p(x, \omega)e^{-2\pi i \left(\frac{h}{a}\cdot x +\frac{k}{b}\cdot \omega\right)}\, dx\, d \omega\,\,.
\end{equation}
\end{example}

\section{Continuity}\label{CONT}
We say that a Banach space $\mathcal S(\mathbb R^{d})\hookrightarrow X\hookrightarrow \mathcal S' (\mathbb R^d)$, with $\mathcal S(\mathbb R^d)$ dense in $X$,  is  \textit{time frequency shifts invariant} (tfs invariant from now on)  if for some polynomial weight function $v$ and $C>0$
\begin{equation}\label{PST8}
\Vert \pi_{z} u\Vert_X\leq C v(z)\Vert u\Vert_X, \quad u\in X,\quad z\in \mathbb R^{2d}.
\end{equation}

\begin{example}\label{EXINV}

\begin{itemize}

\item The $m-$weighted modulation spaces $M^{p,q}_m(\mathbb R^d)$, $p, q\in [1, + \infty]$ are time frequency shifts invariant, see \cite[Theorem 11.3.5] {GRO1}.

\item  The $m-$weighted Lebesgue space $L^p_m(\mathbb R^{d})$ and $m-$weighted Fourier-Lebesgue space $\mathcal FL^p_m(\mathbb R^d)$ are respectively defined as the sets of measurable functions and tempered distributions in $\mathbb R^d$, making finite the norms $\Vert f\Vert_{L^p_{m}}:=\Vert m(\cdot,\omega_0)f\Vert_{L^p}$ and $\Vert f\Vert_{\mathcal FL^p_{m}}=\Vert m(x_0,\cdot) \hat f\Vert_{L^p}$, whatever are $(x_0, \omega_0)\in\mathbb R^{2d}$. (Equivalent norms in $L^p_m(\mathbb R^{d})$ and  $\mathcal FL^p_m(\mathbb R^d)$ should correspond to different choices of $(x_0, \omega_0)$. See \cite[Remark 1.1]{PTT1}) 

For $z=(x,\omega)\in \mathbb R^{2d}$, assuming $(x_0,\omega_0)=(0,0)$ and using \eqref{SOPS1},  we compute:
\begin{equation}\label{EXINV1}
\begin{array}{ll}
\Vert\pi_z f\Vert^p_{L^p}=\Vert M_\omega T_x f\Vert^p_{L^p_m}&=\int m(t,0)^p\vert f(t-x)\vert^p\, dt=\int m(t+x,0)^p\vert f(t)\vert^p \, dt\\
&\leq C^pv(x,0)^p\int m(t,0)^p \vert f(t)\vert^p\, dt=C^p v(x,0)^p\Vert f\Vert ^p_{L^p_m}
\end{array}
\end{equation}
and
\begin{equation}\label{EXINV2}
\begin{array}{l}
\Vert \pi_z f\Vert^p_{\mathcal FL^p_m}=\Vert M_\omega T_x f\Vert^p_{\mathcal FL^p_m}=\int m(0,t)^p\vert \widehat{M_\omega T_x f(t)}\vert^p\, dt\\
=\int m(0,t)^p\vert T_\omega M_{-x}\hat f(t)\vert^p \, dt
=\int  m(0,t)^p\vert \hat f(t-\omega)\vert^p\, dt\\
\leq C^p v(0,\omega)^p\int m(0,t)^p \vert \hat f(t)\vert^p\, dt=C^p v(0,\omega)^p\Vert f\Vert ^p_{\mathcal FL^p_m}.
\end{array}
\end{equation}
Then $L^p_m(\mathbb R^d)$ and $\mathcal FL^P_m(\mathbb R^d)$ are time frequency shifts invariant for any $p\in [1, +\infty]$.
\end{itemize}
\smallskip

In both the examples the positive constants $C$ are  directly obtained by \eqref{SOPS1} and depend only on the weights $m$.
\end{example}

\begin{theorem}\label{TEOPST}
Let $X$ be a time frequency shifts  invariant space, $L\in GL(2d)$, $p\in \mathcal D'(\mathbb T^{2d}_L)$. Assume that the Fourier coefficients $c_\kappa(p)$ defined in \eqref{FL6} satisfy,
\begin{equation}\label{TEOPST1}
\Vert c_\kappa(p)\Vert_{\ell^1_{L, v}}:= \sum_{\kappa\in \mathbb Z^{2d}} v\left(\mathcal JL^{-T}\kappa\right)\vert c_\kappa(p)\vert <+\infty.
\end{equation}
 Then  for any $\tau\in [0,1]$ the operator $\textup{Op}_\tau(p)$ is bounded on $X$ and
\begin{equation}\label{TEOPST2}  
\Vert \textup{Op}_\tau(p)\Vert_{\mathcal L(X)}\leq C\Vert c_\kappa(p)\Vert_{\ell^1_{L, v}},
\end{equation}
Where $C$ is the constant in \eqref{PST8}.\\
\end{theorem}
In lattice terms, see \eqref{REMPST1}, we can write
\begin{equation}\label{TEOPST3}
\Vert c_\kappa (p)\Vert_{\ell^1_{L,v}}=\sum_{\mu=\in \Lambda^\perp}\vert \hat p(\mu)\vert v(\mathcal J \mu):=\Vert \hat p(\mu)\Vert_{\ell^1_v},
\end{equation}
where $\mu= L^{-T}\kappa$,  $\kappa\in \mathbb Z^{2d}$.
\begin{proof}
Using Proposition \ref{PROPPST} and in view of the tfs invariance \eqref{PST8} we obtain for $u\in\mathcal S(\mathbb R^d)$
\begin{equation}\label{TEOPST3.1}
\Vert \textup{Op}_\tau(p)u\Vert_X\le \sum_{\kappa \in \mathbb Z^{2d}}\vert c_\kappa(p)\vert\Vert \pi_{\mathcal J L^{-T}\kappa}u\Vert_X\le 
 C\sum_{\kappa \in \mathbb Z^{2d}}\vert c_\kappa(p)\vert v(\mathcal JL^{-T}\kappa)\Vert u\Vert_X,
\end{equation}
where $C$ is the constant in \eqref{PST8}. The proof follows from the density of $\mathcal S(\mathbb R^d)$ in $X$.
\end{proof}

\subsection{The case of Fourier Multipliers}
Assume now that the symbol is independent of $x$, namely consider a Fourier multiplier  $\sigma=\sigma(\omega)\in\mathcal S^{\prime}(\mathbb R^d)$, $P-${\em periodic}, with  $P\in GL(d)$, that is
\begin{equation}\label{FM_eqt:1}
T_{Pk}\sigma=\sigma\,,\quad\forall\,k\in\mathbb Z^{d}\,.
\end{equation}
In such a case the related pseudodifferential operator, as a linear bounded operator from $\mathcal S(\mathbb R^d)$ to $\mathcal S^{\prime}(\mathbb R^d)$, reads as
\begin{equation}\label{FM_eqt:2}
\sigma(D)u=\mathcal F^{-1}(\sigma\hat u)\,,\quad\forall\,u\in\mathcal S(\mathbb R^d)\,.
\end{equation}
Inserting within \eqref{FM_eqt:2} the Fourier expansion of $\sigma$
\begin{equation}\label{FM_eqt:3}
\sigma(\omega)=\sum\limits_{k\in\mathbb Z^d}c_k(\sigma)e^{2\pi i P^{-T}k\cdot\omega}\,,
\end{equation}
where the series is convergent in $\mathcal S^\prime(\mathbb R^d)$,  we compute
\begin{equation}\label{FM_eqt:4}
\sigma(D)u=\mathcal F^{-1}(\sigma\hat u)=\sum\limits_{k\in\mathbb Z^d}c_k(\sigma)T_{-P^{-T}k}u\,,\quad\mbox{for any}\,\,u\in\mathcal S(\mathbb R^d)\,, 
\end{equation}
with convergence of the series in $\mathcal S^\prime(\mathbb R^d)$; in \eqref{FM_eqt:2}, \eqref{FM_eqt:3}, $c_k(\sigma)$ stand as usual for the Fourier coefficients of $\sigma$. 
\newline 
The following result shows that in the case of Fourier multiplier  operators the sufficient boundedness condition given in Theorem \ref{TEOPST} is also necessary for $\sigma(D)$ to be extended as a linear bounded operator in the weighted Lebesgue space $L^1_v(\mathbb R^d)$, where $v=v(x)$ is a polynomial weight function in $\mathbb R^d$.
\begin{proposition}\label{FM_prop:1}
Let  $\sigma=\sigma(\omega)\in\mathcal S^{\prime}(\mathbb R^d)$ be $P-${\em periodic} for $P\in GL(d)$. If we assume that $\sigma(D)$ extends to a bounded operator in $L^1_v(\mathbb R^d)$, that is
\begin{equation}\label{FM_eqt:6} 
\Vert\sigma(D)u\Vert_{L^1_v}\le C\Vert u\Vert_{L^1_v}\,,\quad\forall\,u\in\mathcal S(\mathbb R^d)\,,
\end{equation}
for a constant $C>0$, then
\begin{equation}\label{FM_eqt:7}
\sum_{k\in \mathbb Z^{d}} v\left(P^{-T}k\right)\vert c_k(\sigma)\vert <+\infty.
\end{equation} 
\end{proposition}
\begin{proof}
It is enough to evaluate $\sigma(D)$ on a non negative continuous function $\tilde u\in L^1_v(\mathbb R^d)$ supported on the compact set $\mathcal P_0:=P^{-T}([0,1]^d)$, such that $\Vert u\Vert_{L^1_v}=1$ \footnote{Such a function $u$ can be defined by $\tilde u(x):=\vert\mbox{det}P\vert\frac{\psi(P^Tx)}{v(x)}$, where $\psi$ is any non negative smooth function supported on the unit cube $Q=[0,1]^d$ such that $\int\psi(z)dz=1$.}. The function $T_{-P^Tk}\tilde u$ will be then supported on $\mathcal P_k:=\mathcal P_0-P^{-T}k$ for all $k\in\mathbb Z^d$. Since the set collection $\{\mathcal P_k\}_{k\in\mathbb Z^d}$ defines a covering on $\mathbb R^d$ such that $\mathcal P_k\cap\mathcal P_h$ has zero Lebesgue measure, whenever $k\neq h$, we get
\begin{equation}\label{FM_eqt:8}
\Vert\sigma(D)\tilde u\Vert_{L^1_v}=\sum\limits_{k\in\mathbb Z^d}\int_{\mathcal P_k}v(x)\vert\sigma(D)\tilde u(x)\vert dx\,.
\end{equation}
It is even clear that $\sigma(D)\tilde u$ reduces to $c_k(\sigma)T_{-P^{-T}k}\tilde u$ in the interior of the set $\mathcal P_k$, for each $k$, so that by change of integration variable and sub-multiplicativity of $v$, we get
\begin{equation}\label{FM_eqt:9}
\begin{split}
\Vert\sigma(D)\tilde u\Vert_{L^1_v}=\sum\limits_{k\in\mathbb Z^d}\vert c_k(\sigma)\vert\int_{\mathcal P_k}v(x)\vert\tilde u(x+P^{-T}k)\vert dx\\
=\sum\limits_{k\in\mathbb Z^d}\vert c_k(\sigma)\vert\int_{\mathcal P_0}v(y-P^{-T}k))\vert\tilde u(y)\vert dy\ge\frac1{K}\sum\limits_{k\in\mathbb Z^d}\vert c_k(\sigma)\vert v(P^{-T}k)\,,
\end{split}
\end{equation} 
with $K>0$ depending only on $P$ and $v$.
\end{proof}
\begin{remark}
Combining Proposition \ref{FM_prop:1} and Theorem \ref{TEOPST} we obtain that, in the case of Fourier multipliers, condition \eqref{FM_eqt:7} is actually equivalent to the continuity in any tfs invariant Banach space, as defined in \eqref{PST8};  here translation invariance of the space is enough, due to the Fourier multiplier structure \eqref{FM_eqt:4}.\\
Instead, condition \eqref{TEOPST1} is no longer necessary for continuity of periodic pseudodifferential operators, with $x$-dependent symbol, in tfs invariant spaces. It can be easily shown by taking a symbol of the following form $p(x,\omega)=\nu(x)\sigma(\omega)$ where $\nu$ is a function in $L^\infty({\mathbb T})$, such that the sequence of its Fourier coefficients $\{c_k(\nu)\}\notin \ell^1(\mathbb Z)$ and $\sigma(\omega)$ satisfies  \eqref{FM_eqt:7}. For instance we could take $\nu(x)=1$ for $0\le x< 1/2$, $\nu(x)=0$ for $1/2\le x < 1$, repeated by periodicity. 
It is straightforward to check that the pseudodifferential operator $p(x,D)$ maps continuously $L^p_v(\mathbb R)$ into itself, whenever $1\le p<+\infty$. On the  other hand we compute at once that
\begin{equation}\label{FM_eqt:10}
 \sum_{(h,k)\in \mathbb Z^{2}}v(k) \vert c_{(h,k)}(p)\vert=\sum_{h\in \mathbb Z}\vert c_h(\nu)\vert \sum_{k\in \mathbb Z}v(k)\vert c_k(\sigma)\vert=+\infty.
\end{equation}

\end{remark}

\section{Invertibility}\label{INVERT}
For the study of invertibility condition of pseudodifferential operators we will make use of the well known properties of the von Neumann series in
Banach algebras, see e.g. \cite{RUD1},  in the following version.
\begin{proposition}\label{ALGEBRA}
Consider $x\in\mathcal A$, where  $(\mathcal A, \Vert \cdot\Vert)$
is a Banach algebra on the field of complex numbers, with
multiplicative identity $e$. If there exists $c\in \mathbb
C\setminus\{0\}$ such that $\Vert e-cx\Vert<1$, then $x$ is invertible in
$\mathcal A$ and
\begin{equation}\label{ALGEBRAINV}
x^{-1}=c \sum_{n=0}^\infty (e-cx)^n.
\end{equation}
\end{proposition}
\begin{theorem}\label{TEOINV}
Let $X$ be a tfs invariant  space, $L\in GL(2d)$, $p\in \mathcal D' (\mathbb T^{2d}_L)$. Assume that the Fourier coefficients $c_\kappa(p)$, $\kappa\in \mathbb Z^{2d}$, satisfy
\begin{equation}\label{TEOINV1}
c_0(p)\neq 0\quad\text{and}\quad
\sum_{\kappa\in \mathbb Z^{2d}_0}\vert c_\kappa(p)\vert v\left(\pi_{\mathcal J L^{-T}\kappa}\right) <\frac{\vert c_0(p)\vert}{C},
\end{equation} 
where $C$ is the constant  in \eqref{PST8}. Then for any $0\leq \tau\leq 1$ 
\begin{itemize}
\item [i)]
the operator $\textup{Op}_\tau(p)$ is invertible in $\mathcal L(X)$;
\item[ii)]
the  norm in $\mathcal L(X)$ of the inverse operator satisfies the following estimate
\begin{equation}\label{TEOINV2}
\Vert (\textup{Op}_\tau(p))^{-1}\Vert_{\mathcal L(X)}\leq \dfrac{1}{\left(1+Cv(0)\right)\vert c_0(p)\vert-C\Vert c_k(p)\Vert_{\ell^1_{L,m}}}.
\end{equation} 
\end{itemize}
\end{theorem} 
Notice that, according to the previous estimate, the invertibility of $\textup{Op}_\tau(p)$ is independent of the quantization $\tau$.

\begin{proof}
Our goal is to estimate the operator norm $\Vert I-c\, \textup{Op}_\tau (p)\Vert _{\mathcal L(X)}$, for any $0\le\tau\le 1$, and $c$ suitable non vanishing constant. 
Let us consider the  definition of Fourier coefficient \eqref {FL6}. Since the \textit{monochromatic signals}  $e^{-2\pi i \langle L^{-T}\kappa, x\rangle}$  are $L$-periodic, it easily follows that $c_0(1)=1$ and $c_\kappa(1)=0$ when $\kappa\in\mathbb Z^{2d}_0$ . Thus $c_\kappa(1-c p)=-c c_\kappa(p)$, when $\kappa\neq 0$ and $c_0(1-c p)=1-c\,c_0(p)$. Assuming that $\langle p, 1\rangle _{\mathbb T_L^{2d}}\neq 0$, thus $c_0(p)= \frac{\langle p, 1\rangle_{\mathbb T^{2d}_L}}{\textup {det} L}\neq 0$,  and setting $c=\frac1{c_0(p)}$, the symbol of the operator $I- \frac{1}{c_0(p)}\textup{Op}_\tau(p)$ admits the following Fourier coefficients
\begin{equation}\label{INV1}
c_{\kappa}\left(1- \frac{p}{c_0(p)}\right)=
\left\{
\begin{array}{ll}
0 & \text {when}\,\,\kappa=0\\
-\frac{c_{\kappa}(p)}{c_0(p)} &\text {when}\,\,\kappa\neq 0\, .
\end{array}
\right.
\end{equation}
The following estimate then follows directly from Theorem \ref{TEOPST},
\begin{equation}\label{INV2}
\begin{array}{ll}
\Vert I- \frac{1}{c_0(p)}\textup{Op}_\tau(p)\Vert_{\mathcal L(X)} & \leq C\Vert  c_k(1- \frac{p}{c_0(p)})\Vert_{\ell^1_{L, m}}=\\
\\
& =\frac{C}{\vert c_0(p)\vert} \sum\limits_{\kappa\in \mathbb Z^{2d}_0} \vert c_\kappa (p)\vert v\left(\mathcal JL^{-T}\kappa\right),
\end{array}
\end{equation}
 where $C$ is the constant in \eqref{PST8}. Thus i) directly follows  from Proposition \ref{ALGEBRA} .\\
Thanks to the assumption \eqref{TEOINV1},  the estimate \eqref{INV2} and Proposition \ref{ALGEBRA},  the inverse operator $(\textup{Op}_\tau(p))^{-1}$ can be expanded in Neumann series
\begin{equation}\label{INV3}
(\textup{Op}_\tau(p))^{-1}= \frac{1}{c_0(p)}\sum_{n=0}^{+\infty}\left(I-\frac{1}{c_0(p)}\textup{Op}_\tau(p)\right)^n,
\end{equation}
then using again \eqref{INV2} we have
\begin{equation}
\begin{array}{ll}
\Vert (\textup{Op}_\tau(p))^{-1}\Vert_{\mathcal L(X)}&\le \frac{1}{\vert c_0(p)\vert}\sum\limits_{n=0}^{+\infty}\Vert I-\frac{1}{c_0(p)}\textup{Op}_\tau(p) \Vert_{\mathcal L(X)}^n\\
\\
&\le \frac{1}{\vert c_0(p)\vert}\sum\limits_{n=0}^{+\infty}\frac{C^n}{\vert c_0(p)\vert^n} \left(\sum\limits_{\kappa\in\mathbb Z^{2d}_0}\vert c_k(p)\vert v\left(\mathcal JL^{-T}\kappa\right)\right)^n\\
&\leq \frac{1}{\vert c_0(p)\vert}\sum\limits_{n=0}^{+\infty}\left(\frac{C}{\vert c_0(p)\vert}(\Vert c_\kappa(p)\Vert_{\ell^1_{L,m}}-\vert c_0(p)\vert v(0))\right)^n\\
\\
&\leq \frac{1}{\vert c_0(p)\vert}\dfrac{1}{1-\frac{C}{\vert c_0(p)\vert}\left(\Vert c_\kappa(p)\Vert_{\ell^1_{L,m}}-\vert c_0(p)\vert v(0)\right)}\\
\\
&=\dfrac{1}{\left(1+Cv(0)\right)\vert c_0(p)\vert-C\Vert c_k(p)\Vert_{\ell^1_{L,m}}},
\end{array} 
\end{equation}
which proves ii).
\end{proof}

\appendix
\section{On periodic distributions and distributions on the torus}\label{PER-DIST}
This section is devoted to shortly review some known facts about the comparison between the space of {\em periodic} distributions in $\mathbb R^n$, see Section \ref{SubPD}, and the space $\mathcal D^\prime(\mathbb T^n_L)$ of {\em distributions on the torus} $\mathbb T^n_L:=\mathbb R^n/L\mathbb Z^n$, in order to justify the identification of the aforementioned spaces, that we have implicitly assumed in the whole  paper. Recall that $\mathcal D^\prime(\mathbb T^n_L)$ is the space linear continuous forms on the function space $C^\infty(\mathbb T^n_L)$ (the latter being endowed with its natural Fr\'echet space topology). As already mentioned in Section \ref{SubPD}, the reader is referred to Ruzhansky - Turunen \cite{RuzTur} for a thorough study of distributions on the torus.
\newline
For the rest, the results collected herebelow come essentially from making  explicit some of the results established in H\"{o}rmander \cite[Section 7.2]{HOR0}.
\newline
It is well understood that functions on the torus $\mathbb T^n_L$ can be naturally identified with $L-$periodic functions in $\mathbb R^n$. Below, we will illustrate a way to extend the same identifications to all $L-$periodic distributions in $\mathbb R^n$. This extension to distributions can be made by a duality argument, as it is customary. Thought the following arguments work in the case of $L-$periodicity, with arbitrary invertible matrix $L\in GL(n)$, just for simplicity we will restrict to the case of $L=I_n$ the $n\times n$ identity matrix, leading to $1-$periodic functions and distributions.
\newline
So let us first consider a $1-$periodic measurable function $f=f(x)$ in $\mathbb R^n$ such that $f\in L^1([0,1]^n)$; of course, such a function $f$ is a $1-$periodic tempered distribution in $\mathbb R^n$. On the other hand, when  identified (as usual) with an integrable function on the torus $\mathbb T^n$, $f$ defines an element of $\mathcal D^\prime(\mathbb T^n)$, whose action  on test functions $\psi\in C^\infty(\mathbb T^n)$ is given by
\[
\langle f,\psi\rangle_{\mathbb T^n}:=\int_{[0,1]^n}f(x)\psi(x)dx\,;
\]
in order to avoid confusion, here and below $\langle\cdot,\cdot\rangle_{\mathbb T^n}$ stands for the duality pair between $\mathcal D^\prime(\mathbb T^n)$ and $\mathcal C^\infty(\mathbb T^n)$, whereas $\langle\cdot,\cdot\rangle$ is denoting the dual pair between $\mathcal S(\mathbb R^n)$ and $\mathcal S^\prime(\mathbb R^n)$.
\newline
Testing $f$ against an arbitrary function $\varphi\in\mathcal S(\mathbb R^n)$ we compute
\begin{equation}\label{APP0}
\begin{split}
\langle f,\varphi\rangle&=\int_{\mathbb R^n}f(x)\varphi(x)dx=\sum\limits_{\kappa\in\mathbb Z^n}\int_{[0,1]^n+\kappa}f(x)\varphi(x)dx\\
&=\sum\limits_{\kappa\in\mathbb Z^n}\int_{[0,1]^n}f(y+\kappa)\varphi(y+\kappa)dy=\int_{[0,1]^n}\sum\limits_{\kappa\in\mathbb Z^n}f(y+\kappa)\varphi(y+\kappa)dy\\&=\int_{[0,1]^n}f(y)\sum\limits_{\kappa\in\mathbb Z^n}\varphi(y+\kappa)dy\,,
\end{split}
\end{equation}
where the countable-additivity of the Lebesgue integral, together with Fubini's theorem to interchange the sum and the integral, and the periodicity of $f$ are used.
\newline
The function 
$$
\varphi_{\rm per}(x):=\sum\limits_{\kappa\in\mathbb Z^n}\varphi(x+\kappa)\,,\quad x\in\mathbb R^n\,,
$$ 
which appears in the last line in \eqref{APP0}, is a $1-$periodic $C^\infty$function that can be canonically identified with a (unique) element in $C^\infty(\mathbb T^n)$ ); we call it  $1-$\textit{periodization} of $\varphi$. 
It is fairly easy to check that the mapping $\varphi\mapsto\varphi_{\rm per}$ is continuous as a (linear) operator from $\mathcal S(\mathbb R^n)$ to $C^\infty(\mathbb T^n)$ (thus from $C^\infty_0(\mathbb R^n)$ to $C^\infty(\mathbb T^n)$, as well). Therefore the calculations above show naturally the way to define a linear mapping $\mathcal T$ from the space $\mathcal D^\prime(\mathbb T^n)$  to the subspace of $\mathcal D^\prime(\mathbb R^n)$ consisting of $1-$periodic distributions (which are  automatically  tempered distributions), just by setting for $U\in\mathcal D^\prime(\mathbb T^n)$
\begin{equation}\label{APP2}
\langle\mathcal T(U),\varphi\rangle:=\langle U,\varphi_{\rm per}\rangle_{\mathbb T^n}\,,\quad\forall\,\varphi\in\mathcal S(\mathbb R^n)\,.
\end{equation}
It is easy to verify that $\mathcal T$ acts continuously from $\mathcal D^\prime(\mathbb T^n)$ to $1-$periodic distributions in $\mathbb R^n$. Moreover  \eqref{APP0} shows that $\mathcal T(U)$ properly reduces to the $1-$periodic tempered distribution in $\mathbb R^n$ corresponding to a function $U\in L^1(\mathbb T^n)$. 
\newline
It is a little less obvious that $\mathcal T$ is invertible, so that it actually defines an isomorphism. This can be proved by noticing that every function $\psi\in C^\infty(\mathbb T^n)$ can be regarded as the $1-$periodization of (at least) one function $\varphi\in\mathcal S(\mathbb R^n)$, namely $\psi=\varphi_{\rm per}$. To see this, consider a function $\phi\in C^\infty_0(\mathbb R^n)$ satisfying \eqref{PHI} and set
\begin{equation}\label{APP3}
\varphi:=\phi\psi\,,
\end{equation}
for any function $\psi\in C^\infty(\mathbb T^n)$ (identified with its $1-$periodic $C^\infty-$counterpart in $\mathbb R^n$). Of course, $\varphi$ defined above belongs to $C^\infty_0(\mathbb R^n)$ so it is rapidly decreasing in $\mathbb R^n$; moreover, in view of \eqref{PHI} and the periodicity of $\psi$, we get for any $x\in\mathbb R^n$
\begin{equation*}
\varphi_{\rm per}(x)=\sum\limits_{\kappa\in\mathbb Z^n}\varphi(x+\kappa)=\sum\limits_{\kappa\in\mathbb Z^n}\phi(x+\kappa)\psi(x+\kappa)=\psi(x)\sum\limits_{\kappa\in\mathbb Z^n}\phi(x+\kappa)=\psi(x)\,,
\end{equation*}
showing that $\psi$ is actually the $1-$periodization of $\varphi$. This leads to associate to any periodic distribution $u\in\mathcal S^\prime(\mathbb R^n)$ a linear form $U$ on $C^\infty(\mathbb T^n)$ by setting for every $\psi\in C^\infty(\mathbb T^n)$
\begin{equation}\label{APP4}
\langle U,\psi\rangle_{\mathbb T^n}:=\langle u,\varphi\rangle\,,
\end{equation}  
being $\varphi=\varphi(x)$ the rapidly decreasing (actually compactly supported smooth) function in $\mathbb R^n$ associated to $\psi$ as in \eqref{APP3}. 
In order to give consistency to the definition of $U$, we must prove that it is independent of $\varphi$. Let us notice that, thanks to \eqref{HS2}, any function $\varphi\in\mathcal S(\mathbb R^n)$, whose $1-$periodization is given by $\psi\in C^\infty(\mathbb T^n)$, satisfies the following:
\[
\hat\varphi(\kappa)=c_\kappa(\psi)\,,\quad\forall\,\kappa\in\mathbb Z^n\,.
\]
Thus using the Fourier expansion \eqref{HS3} we obtain 
\[
\langle u,\varphi\rangle=\sum_{\kappa\in\mathbb Z^n} c_\kappa \langle e^{2\pi i \langle \cdot, \kappa\rangle},\varphi\rangle=\sum_{\kappa\in\mathbb Z^n} c_\kappa\hat\varphi(\kappa)=\sum_{\kappa\in\mathbb Z^n} c_\kappa c_\kappa(\psi)\,.
\]
This shows the consistency of the definition of $U$ above; the continuity of the linear form $U$ on $C^\infty(\mathbb T^n)$ also easily follows, so that $U\in\mathcal D^\prime(\mathbb T^n)$. 
\newline
The mapping $u\mapsto U$ defined on $1-$periodic distributions in $\mathbb R^n$ by \eqref{APP4} provides a linear continuous operator from the space of  
$1-$periodic (tempered) distributions in $\mathbb R^n$ to $\mathcal D^\prime(\mathbb T^n)$, which is actually the inverse of the operator $\mathcal T$ introduced before, see \eqref{APP2}. This equivalently shows that $\mathcal T$ is an isomorphism, up to which $1-$periodic distributions in $\mathbb R^n$ can be thought to be elements of $\mathcal D^\prime(\mathbb T^n)$ and viceversa.
\newline
Thus a $1-$periodic distribution $u\in\mathcal S^\prime(\mathbb R^n)$ can be regarded as a linear continuous form on $C^\infty(\mathbb T^n)$; in particular this makes rigorous the testing of $u$ against $1-$periodic smooth functions in $C^\infty(\mathbb T^n)$, such as $e^{-2\pi i\langle\kappa,\cdot\rangle}$, with $\kappa\in\mathbb Z^n$, providing meanwhile an explicit explanation of formula \eqref{HS4} for the {\em Fourier coefficients} of $u$.

\section*{Acknowledgments}
The research of A. Morando is partially supported by the Italian MUR Project PRIN prot. 20204NT8W4.\\
 G. Garello is supported by the Local Research Grant of University of Torino.


\end{document}